\documentclass[12pt]{article}

\usepackage{graphicx}            
\usepackage{amsfonts}  
\usepackage[breaklinks]{hyperref}

\usepackage{tikz}
\usepackage{float}
\usetikzlibrary{patterns,arrows,decorations.pathreplacing}
\usetikzlibrary{decorations.markings}
\usetikzlibrary{calc,shapes}
\usepackage{calc}
\usepackage{ifthen}

\tikzstyle{red_vertex}=[circle,fill=red!25,minimum size=20pt,inner sep=0pt]
\tikzstyle{blue_vertex}=[circle,fill=blue!25,minimum size=20pt,inner sep=0pt]
\tikzstyle{vertex}=[circle,fill=black,minimum size=5pt,inner sep=0pt]
\tikzstyle{edge} = [draw,line width=2pt,-]
\tikzstyle{highlighted_edge} = [draw, line width =10pt,-]

\usepackage{pgfmath}
\pgfdeclarelayer{bg}    
\pgfsetlayers{bg,main}  

\newcommand{\s}[1]{\left\lvert #1 \right\rvert}

\newcommand{\eps}{\varepsilon}

\usepackage{amsmath,amsthm,amssymb,enumerate,mathrsfs,mathtools}
\usepackage{relsize}
\usepackage{a4wide}
\usepackage{verbatim}
\usepackage{xcolor}

\usepackage{placeins} 

\usepackage[old]{old-arrows} 

\usepackage{makecell}
\usepackage[numbers,sort&compress]{natbib}

\usepackage{enumitem}
\setlistdepth{9} 
\renewlist{enumerate}{enumerate}{9}
\setlist[enumerate,1]{label=(\arabic*)}
\setlist[enumerate,2]{label=(\Roman*)}
\setlist[enumerate,3]{label=(\Alph*)}
\setlist[enumerate,4]{label=(\roman*)}
\setlist[enumerate,5]{label=(\alph*)}
\setlist[enumerate,6]{label=(\arabic*)}
\setlist[enumerate,7]{label=(\Roman*)}
\setlist[enumerate,8]{label=(\Alph*)}
\setlist[enumerate,9]{label=(\roman*)}

\usepackage[capitalise]{cleveref}
\crefname{equation}{}{}
\crefname{prop}{Proposition}{Propositions}
\crefname{enumi}{}{}
\crefname{figure}{Figure}{Figures}

\usepackage[utf8]{inputenc}

\newtheorem{theorem}{Theorem}[section]
\crefname{theorem}{Theorem}{Theorems}
\newtheorem{lemma}[theorem]{Lemma}
\newtheorem{proposition}[theorem]{Proposition}
\newtheorem{corollary}[theorem]{Corollary}

\newtheorem{claim}[theorem]{Claim}

\theoremstyle{definition}

\numberwithin{equation}{section}

\DeclareMathOperator{\blue}{blue}
\DeclareMathOperator{\red}{red}
\def\al#1{}
	\renewcommand{\al}[1]{\footnote{\textbf{AL: }#1}}  
	
\def\vp#1{}
	\renewcommand{\vp}[1]{\footnote{\textcolor{green!40!black}{\textbf{VP: }#1}}} 
    
\newenvironment{marknew}{\color{red} }{ }
\newcommand{\mn}{\begin{marknew}}
\newcommand{\umn}{\end{marknew}}

\newenvironment{proofclaim}[1][Proof of Claim]{\begin{proof}[#1]}{\end{proof}}

\usepackage[UKenglish]{babel}
\usepackage[UKenglish]{isodate}

\newcommand{\splitatcommas}[1]{%
  \begingroup
  \ifnum\mathcode`,="8000
  \else
    \begingroup\lccode`~=`, \lowercase{\endgroup
      \edef~{\mathchar\the\mathcode`, \penalty0 \noexpand\hspace{0pt plus 1em}}%
    }\mathcode`,="8000
  \fi
  #1%
  \endgroup
}

\newcommand{\crel}[1]{%
  \global\setbox1=\hbox{$#1$}%
  \global\dimen1=0.5\wd1
  \mathrel{\hbox to\dimen1{$#1$\hss}}&\mathrel{\mspace{-\thickmuskip}\hbox to\dimen1{}}%
}

\def\COMMENT#1{}
\renewcommand{\COMMENT}[1]{\footnote{#1}}  

\newcommand{\EMAIL}[1]{\textit{{E-mail}}: \texttt{\href{mailto:#1}{#1}}} 

\usepackage{lineno}
\newcommand*\patchAmsMathEnvironmentForLineno[1]{%
  \expandafter\let\csname old#1\expandafter\endcsname\csname #1\endcsname
  \expandafter\let\csname oldend#1\expandafter\endcsname\csname end#1\endcsname
  \renewenvironment{#1}%
     {\linenomath\csname old#1\endcsname}%
     {\csname oldend#1\endcsname\endlinenomath}}%
\newcommand*\patchBothAmsMathEnvironmentsForLineno[1]{%
 \patchAmsMathEnvironmentForLineno{#1}%
  \patchAmsMathEnvironmentForLineno{#1*}}%
\AtBeginDocument{%
\patchBothAmsMathEnvironmentsForLineno{equation}%
\patchBothAmsMathEnvironmentsForLineno{align}%
\patchBothAmsMathEnvironmentsForLineno{flalign}%
\patchBothAmsMathEnvironmentsForLineno{alignat}%
\patchBothAmsMathEnvironmentsForLineno{gather}%
\patchBothAmsMathEnvironmentsForLineno{multline}%
}

\title{Almost partitioning every $2$-edge-coloured complete $k$-graph into~$k$ monochromatic tight cycles} 
\author{
Allan Lo 
\thanks{School of Mathematics, University of Birmingham, UK, \EMAIL{s.a.lo@bham.ac.uk}.} 
\and Vincent Pfenninger \thanks{Graz University of Technology, Institute of Discrete Mathematics, Steyrergasse 30, 8010 Graz, Austria, \EMAIL{pfenninger@math.tugraz.at}.}}
\date{\today}

\begin{document}

\maketitle

\begin{abstract}
    A \emph{$k$-uniform tight cycle} is a $k$-graph with a cyclic order of its vertices such that every~$k$ consecutive vertices from an edge. We show that for $k\geq 3$, every red-blue edge-coloured complete $k$-graph on~$n$ vertices contains~$k$ vertex-disjoint monochromatic tight cycles that together cover $n - o(n)$ vertices.
\end{abstract}

\section{Introduction}

Monochromatic partitioning is an area of combinatorics that has its origin in a remark of Gerencs\'er and Gy\'arf\'as~\cite{Gerencser1967} that any $2$-edge-colouring of the complete graph~$K_n$ contains a spanning path that consists of a red\footnote{We assume that the colours in a $2$-edge-colouring are always red and blue.} path followed by a blue path. In particular, every $2$-edge-coloured~$K_n$ admits a partition of its vertex set into a red path and a blue path. In a subsequent paper, Gy\'arf\'as~\cite{Gyarfas1983} proved the stronger result that every $2$-edge-coloured~$K_n$ contains a red cycle and a blue cycle that share at most one vertex and together cover all vertices. Lehel's conjecture states that in fact one can do even better. Lehel conjectured that every $2$-edge-coloured~$K_n$ can be partitioned into a red cycle and a blue cycle. Here the empty set, a single vertex, and a single edge are considered to be cycles. Lehel's conjecture was first proved for large enough~$n$ by {\L}uczak, Rödl and Szemer{\'e}di~\cite{Luczak1998} using the Regularity Lemma. Allen~\cite{Allen2008} later gave a different proof of this that does not use the Regularity Lemma, thus improving the bound on~$n$. Bessy and Thomass\'e~\cite{Bessy2010} finally gave a short and elegant proof of Lehel's conjecture for all~$n$.

Related problems have also been studied in the case where instead of~$K_n$ we have a graph~$G$ with a minimum degree condition~\cite{Balogh2014,DeBiasio2017,Letzter2019,Pokrovskiy2023,Korandi2021} or~$K_{n,n}$~\cite{Pokrovskiy2014,Stein2022,Gyarfas1973,Gyarfas1983}. For related problems with more colours see~\cite{Erdos1991,Gyarfas2006,Pokrovskiy2014,Pokrovskiy2023,Korandi2021}. See~\cite{Gyarfas2016} and Sections~7.4 and~9.7 of~\cite{Simonovits2019} for surveys on monochromatic partitioning.

Here our interest lies in generalisations of Lehel's conjecture to hypergraphs and tight cycles (for related problems in hypergraphs about other types of cycles see~\cite{gyarfas2013,Sarkozy2014,Bustamante2018}).
For $k \ge 2$, a \emph{$k$-graph} (or \emph{$k$-uniform hypergraph})~$H$ is a pair of sets $(V(H), E(H))$ such that $E(H) \subseteq \binom{V(H)}{k}$.\footnote{For a set~$X$, $\binom{X}{k}$ denotes the set of all subsets of~$X$ of size~$k$.}
A \emph{$k$-uniform tight cycle} is a $k$-graph with a cyclic order of its vertices such that its edges are exactly the sets of~$k$ consecutive vertices in the order. In this paper, any set of at most~$k$ vertices is also considered a tight cycle.

The problem we want to consider is how to cover almost all vertices of every $2$-edge-coloured complete $k$-graph with as few vertex-disjoint monochromatic tight cycles as possible. 
For $k =3$, Bustamante, H\`an and Stein~\cite{Bustamante2019} proved that in every $2$-edge-coloured complete $3$-graph almost all vertices can be partitioned into a red and a blue tight cycle. Subsequently, Garbe, Mycroft, Lang, Lo, and Sanhueza-Matamala~\cite{Garbe2023} showed that in fact there is a partition of all the vertices into two monochromatic tight cycles. Here it is necessary to allow the two monochromatic tight cycles to possibly have the same colour. 
Indeed for every $k \geq 3$, there exists a complete $k$-graph on arbitrarily many vertices that does not admit a partition of its vertices into a red and a blue tight cycle~\cite[Proposition~1.1]{Lo2023}. In an earlier paper~\cite{Lo2023}, the authors proved that for $k = 4$ it is also possible to almost partition the vertices of every $2$-edge-coloured $4$-graph into a red and a blue tight cycle. In the same paper, the authors also showed a weaker result for $k = 5$, that~$4$ vertex-disjoint monochromatic tight cycles suffice to cover almost all vertices of every $2$-edge-coloured complete $5$-graph. The only bound for general~$k$ that was known is given by a result of Bustamante, Corsten, Frankl, Pokrovskiy and Skokan~\cite{Bustamante2020}. They showed that every $r$-edge-coloured complete $k$-graph on~$n$ vertices can be partitioned into~$c(r,k)$ monochromatic tight cycles. However, the constant~$c(r,k)$ that can be obtained from their proof is very large. 
Indeed to cover almost all vertices, they simply repeatedly find a monochromatic tight cycle using the fact that the Ramsey number for the $k$-uniform tight cycle on~$N$ vertices is linear in~$N$.

Our aim here is to show a reasonable general bound on the number of tight cycles that are needed to almost partition a $2$-edge-coloured complete $k$-graph on $n$ vertices, that is, covering all but at most $\eps n$ vertices. 
A common approach to this problem is to apply a Ramsey-type result to find a long monochromatic tight cycle and then delete the vertex-set of this cycle and repeat. 
Allen, B\"ottcher, Cooley and Mycroft~\cite{Allen2017} showed that one can always find a monochromatic tight cycle on $( 1/2-o(1) ) n $ vertices. 
On the other hand, Alon, Frankl and Lov\'asz~\cite{AFL} showed that one can only guarantee a monochromatic matching of size at most~$n/(k+1)$ and so a monochromatic tight cycle on at most $kn/(k+1)$ vertices. 
Thus, this method always gives a number of tight cycles dependent on $\eps$, namely at least $\log(1/\eps)/\log(k+1)$ tight cycles.

In this paper, we give an upper bound on the number of cycles needed that is independent of~$\eps$. 
In particular, we show that $k$ tight cycles suffice.
 We remark that this result is probably not tight. The only known lower bound on the number of tight cycles needed is the trivial lower bound of~$2$.
Furthermore, throughout the paper, we assume that $k$ is fixed and $\eps$ is arbitrarily small. 
Otherwise, say when $k \ge \log (1/\eps)$, the greedy approach mentioned above is better. 

\begin{theorem} \label{thm:main}
    For every $\eps > 0$ and $k \geq 3$, there exists an integer~$n_0$ such that every $2$-edge-coloured complete $k$-graph on $n \geq n_0$ vertices contains~$k$ vertex-disjoint monochromatic tight cycles covering at least~$(1-\eps)n$ vertices.
\end{theorem}

Our proof is based on a hypergraph version of {\L}uczak's connected matching method (see~\cite{Luczak1999} for the original method). Roughly speaking the idea is as follows. Let~$K$ be a $2$-edge-coloured complete $k$-graph on~$n$ vertices. We apply the Regular Slice Lemma (a form of the Hypergraph Regularity Lemma that is due to Allen, B\"ottcher, Cooley and Mycroft~\cite{Allen2017}) to~$K^{\red}$.\footnote{For a $2$-edge-coloured $k$-graph~$H$, we denote by~$H^{\red}$ and~$H^{\blue}$ the subgraph induced by the red edges and the subgraph induced by the blue edges of~$H$, respectively.} 
Since any regularity partition for a $k$-graph is also a regularity partition for its complement, this gives rise to a reduced $k$-graph $\mathcal{R}$ that is a $2$-edge-coloured almost complete $k$-graph. 
So a red edge $i_1 \dots i_k$ in $\mathcal{R}$ means that~$K^{\red}$ is regular with respect to the corresponding clusters $V_{i_1}, \dots, V_{i_k}$ and at least half the edges of~$K$ with one vertex in each~$V_{i_j}$, $j \in [k]$ are red.\footnote{For $n \in \mathbb{N}$, we let $[n] = \{1, \dots, n\}$.} 
It can be shown that there is a red tight cycle in~$K$ that contains almost all of the vertices in $\bigcup_{j \in [k]} V_{i_j}$. The idea is now to combine a matching~$M$ of red edges in $\mathcal{R}$ into an even longer red tight cycle that covers almost all the vertices in the clusters that are covered by~$M$. 
However, in order for this to work we will need to be able to construct a red tight path\footnote{A \emph{tight path} is a $k$-graph with a linear order of its vertices such that every $k$-consecutive vertices form an edge. Or alternatively a tight path is a $k$-graph obtained by deleting a vertex from a tight cycle.} that goes from the clusters corresponding to one edge of~$M$ to the clusters corresponding to another edge of~$M$. So we require our red matching~$M$ in~$\mathcal{R}$ to be `connected' in some sense. 
To this end, we need the following definitions. A \emph{tight pseudo-walk of length~$m$ from~$e$ to~$e'$} in a $k$-graph~$H$ is a sequence of edges $e_1 \dots e_m$ in~$H$ such that $\s{e_i \cap e_{i+1}} \geq k-1$ for each $i \in [m-1]$ and $(e_1,e_m)=(e,e')$. 
A $k$-graph~$H$ is \emph{tightly connected} if for every pair of edges $e, e' \in H$, there is a tight pseudo-walk from~$e$ to~$e'$ in~$H$. 
A \emph{tight component} in a $k$-graph~$H$ is a maximal tightly connected subgraph of~$H$. 
Let~$H$ be a $2$-edge-coloured $k$-graph. 
A \emph{red or blue tight component} in~$H$ is a tight component in~$H^{\red}$ or in~$H^{\blue}$, respectively. 
A \emph{monochromatic tight component} in~$H$ is a red or a blue tight component in~$H$.
The hypergraph version of {\L}uczak's connected matching method that we need now roughly says the following. If $\mathcal{R}$ contains a matching that covers almost all vertices and uses edges from at most~$k$ monochromatic tight components, then there exists~$k$ monochromatic tight cycles in~$K$ that are vertex-disjoint and together cover almost all vertices. See \cref{cor:matchings_to_cycles} for the formal statement. The proof of \cref{thm:main} is now reduced to proving that every almost complete $2$-edge-coloured $k$-graph contains a matching that covers almost all vertices and only uses edges from at most~$k$ monochromatic tight components. This is \cref{lem:main} and is proved in \cref{section:matchings}. We deduce \cref{lem:main} from a novel structural lemma (\cref{lem:blue_tight_walk_2}) about the monochromatic tight components in any $2$-edge-coloured almost complete $k$-graph which we think might be of independent interest. This structural lemma states that if $e_1$ and $e_2$ are two red edges on the boundary of a blue tight component $B$ (meaning that there exist $f_1, f_2 \in B$ with $|e_1 \cap f_1| = |e_2 \cap f_2| = k-1$), then either $e_1$ and $e_2$ are in the same red tight component or every tight pseudo-walk from $e_1$ to $e_2$ contains an edge of $B$.

The remainder of the paper is organised as follows. In \cref{section:preliminaries}, we state the necessary definitions. In \cref{section:matchings}, we prove \cref{lem:main} which shows that $\mathcal{R}$ contains the desired matching that only uses edges from~$k$ monochromatic tight components. In \cref{section:end}, we show that \cref{thm:main} follows easily from \cref{lem:main} using the hypergraph version of {\L}uczak's connected matching method (\cref{cor:matchings_to_cycles}).

\section{Preliminaries} \label{section:preliminaries}
\subsection{Basics}
For $n \in \mathbb{N}$, we let $[n] = \{1, \dots, n\}$. 
Moreover, for a set~$X$ and $i \in \mathbb{N}$, we let $\binom{X}{i} = \{ S \subseteq X \colon \s{S} = i\}$. 
We abuse notation by sometimes writing $x_1 \dots x_j$ for $\{x_1, \dots, x_j\}$.

For a $k$-graph~$H$, we abuse notation by identifying~$H$ with its edge set~$E(H)$. For a set $W \subseteq V(H)$, we denote by~$H[W]$ the $k$-graph with $V(H[W]) = W$ and $E(H[W]) = \{e \in E(H) \colon e \subseteq W\}$. For a set $S \subseteq V(H)$ with $\s{S} \leq k-1$, we let $N_H(S) = \{S' \in \binom{V(H)}{k-\s{S}} \colon S \cup S' \in E(H)\}$ and $d_H(S) = \s{N_H(S)}$.
A $k$-graph~$H$ on~$n$ vertices is called \emph{$(\mu, \alpha)$-dense} if, for each $i \in [k-1]$, we have $d_H(S) \geq \mu \binom{n}{k-i}$ for all but at most $\alpha \binom{n}{i}$ sets $S \in \binom{V(H)}{i}$ and $d_H(S) = 0$ for all other sets~$S \in \binom{V(H)}{i}$.

When we say that a statement holds for constants~$a$ and~$b$ such that $0 < a \ll b <1$, we mean that the statement holds provided that~$a$ is chosen sufficiently small in terms of~$b$. More precisely, we mean that there exists a non-decreasing function $f \colon (0,1) \rightarrow (0,1)$ such that the statement holds for all $a, b \in (0,1)$ with $a \leq f(b)$. Similar \emph{hierarchies} with more variables are defined analogously. Moreover, if~$1/n$ appears in such a hierarchy then we implicitly assume that~$n$ is a positive integer.

\subsection{Tight pseudo-walks and triangulations}

Recall that a \emph{tight pseudo-walk (of length~$m$)} in a $k$-graph~$H$ is a sequence of edges $e_1 \dots e_m$ in~$H$ such that $\s{e_i \cap e_{i+1}} \geq k-1$\footnote{We use the condition $\s{e_i \cap e_{i+1}}\geq k-1$ rather than $\s{e_i \cap e_{i+1}} = k-1$ to allow $e_i = e_{i+1}$.} for every $i \in [m-1]$. A \emph{closed tight pseudo-walk} is a tight pseudo-walk $e_1 \dots e_m$ such that $\s{e_1 \cap e_m} \geq k-1$. 

A \emph{plane graph} is a planar graph together with an embedding in the plane $\mathbb{R}^2$. A \emph{nearly triangulated plane graph} is a plane graph such that every face except possibly the outer face is a triangle. A \emph{walk} in a plane graph~$D$ is a sequence $x_1 \dots x_m$ of vertices of~$D$ such that $x_i x_{i+1} \in E(D)$ for each $i \in [m-1]$. All graphs and in particular all plane graphs in this paper are finite. 

Given a closed tight pseudo-walk $Q = e_1 \dots e_m$ in a $k$-graph~$H$, we say that a nearly triangulated plane graph~$D$ is a \emph{triangulation} of~$Q$ if there exists a map $\varphi \colon V(D) \rightarrow E(H)$ such that if $xy \in E(D)$, then $\s{\varphi(x) \cap \varphi(y)} \geq k-1$ and the outer face of~$D$ is $x_1 \dots x_m$ (in order) where $\varphi(x_i) = e_i$ for each $i \in [m]$. Note that in this setting, any walk in~$D$ corresponds to a tight pseudo-walk in~$H$.

\section{A large matching using edges from few monochromatic tight components} \label{section:matchings}

The aim of this section is to prove the following lemma which says that in any $2$-edge-coloured $(1-\eps, \eps)$-dense $k$-graph we can find a matching that covers almost all vertices and only contains edges from at most~$k$ monochromatic tight components. 

\begin{lemma} \label{lem:main}
    Let $1/n \ll \eps \ll \eta \ll 1/k \leq 1/2$. Let~$H$ be a $2$-edge-coloured $(1-\eps, \eps)$-dense $k$-graph on~$n$ vertices. Then there exists a matching in~$H$ that covers all but at most $\eta n$ vertices of~$H$ and only contains edges from at most~$k$ monochromatic tight components of~$H$.
\end{lemma}

The case $k=3$ is already proved in~\cite{Bustamante2019} (in which they showed that a red and a blue tight component suffice).
The first step of the proof is to find a red and a blue tight component~$R$ and~$B$, respectively, of~$H$ such that almost all $2$-subsets of~$V(H)$ are contained in some edge of $R \cup B$. 
One then finds a large matching in~$R \cup B$.
Thus a natural first step of proving Lemma~\ref{lem:main} is to find a constant number of monochromatic tight components of~$H$, such that almost all $(k-1)$-subsets of~$V(H)$ are contained in some edge of these tight components. 
However this is not possible when $k \ge 4$ as shown by the following example. 
Let $V_1, \dots, V_{\ell}$ be an equipartition of a set of $n$ vertices. 
Consider the $2$-edge-coloured complete $k$-graph on $\bigcup_{i \in [\ell]}V_i$ such that an edge~$e$ is blue if $|e \cap V_i| > k/2$, and red otherwise. 
Observe that there are $\ell$ blue tight components. 
Moreover, each $\binom{V_i}{k-1}$ has to be ``covered'' by a distinct blue tight component.

Instead, we will enlarge our maximal matching as we choose tight components as described in the following proof sketch. 
Consider a monochromatic tight component~$F_*$ in~$H$ and let $\mathcal{G}_1 = \{F_*\}$.
We say that two monochromatic tight components~$F_1$ and~$F_2$ of~$H$ are \emph{adjacent} if they have opposite colours and there are edges $e_1 \in F_1$ and $e_2 \in F_2$ such that $|e_1 \cap e_2| = k-1$.
Now for each $i \geq 2$ in turn, let $\mathcal{G}_i$ be the set of monochromatic tight components that are adjacent to a monochromatic tight component in $\mathcal{G}_{i-1}$ and not already in $\bigcup_{j \in [i-1]} \mathcal{G}_j$. 
Moreover, for each $i \geq 1$, we let $\mathcal{E}(\mathcal{G}_i) = \bigcup_{F \in \mathcal{G}_i} F$, that is, $\mathcal{E}(\mathcal{G}_i)$ is the set of edges that are in some monochromatic tight component $F \in \mathcal{G}_i$.
It is easy to see that all edges of~$H$ are in $\bigcup_{j \in [2k]} \mathcal{E}(\mathcal{G}_j)$ (see Proposition~\ref{prop:tight_walk_bridge}).
In fact, if $F_*$ spans almost all vertices (such an $F_*$ exists by \cref{prop:Almost_spanning_MTC}), then almost all edges of~$H$ are in~$\bigcup_{j \in [k]} \mathcal{E}(\mathcal{G}_j)$.
For simplicity we assume that $H = \bigcup_{j \in [k]} \mathcal{E}(\mathcal{G}_j)$.

We now set $W_0 = V(H)$ and for each $i = 1, \dots, k$ in turn, we let~$M_i$ be a maximal matching in $H[W_{i-1}]\cap \bigcup_{j \in [i]} \mathcal{E}(\mathcal{G}_j)$ and $W_i = W_{i-1} \setminus V(M_i)$.
Since $\bigcup_{j \in [k]} \mathcal{E}(\mathcal{G}_j) = H$ is $(1-\eps, \eps)$-dense, $\bigcup_{j \in [k]} M_j$ is a maximal matching in~$H$ covering almost all vertices of~$H$.

It remains to show that $\bigcup_{j \in [k]} M_j$ is contained in at most $k$ monochromatic tight components. 
Note that, for each $i \in [k+1]$,
\begin{align} \label{no_edges_in_W_i}
    H[W_{i-1}] \cap \bigcup_{j \in [i-1]} \mathcal{E}(\mathcal{G}_j) = \varnothing
\end{align}
by our choices of~$M_{i-1}$ and~$W_{i-1}$. 
Hence $M_i \subseteq H[W_{i-1}] \cap \mathcal{E}(\mathcal{G}_i)$. 
Therefore, it suffices to show that $H[W_{i-1}]\cap \mathcal{E}(\mathcal{G}_i)$ (and so $M_i$) consists of edges from one monochromatic tight component.

Suppose for a contradiction that there are two edges~$e_1$ and~$e_2$ in $H[W_{i-1}]\cap \mathcal{E}(\mathcal{G}_i)$ that are in different monochromatic tight components. 
Suppose further that $|W_{i-1}| \ge \eta n$ (or else $\bigcup_{j \in [i-1]} M_j$ is already an almost perfect matching).
Since each component of $\mathcal{G}_{i}$ is adjacent to some components of~$\mathcal{G}_{i-1}$
and $\mathcal{E}(\mathcal{G}_i)$ is monochromatic, all edges in $\mathcal{E}(\mathcal{G}_i)$ (including $e_1$ and $e_2$) have the same colour, say red. 
In~$H[W_{i-1}]$, there exists a tight pseudo-walk~$P$ from~$e_1$ to~$e_2$. 
Recall that $\bigcup_{j \in [i]} \mathcal{E}(\mathcal{G}_j)$ is tightly connected, so $\bigcup_{j \in [i]} \mathcal{E}(\mathcal{G}_j)$ contains a tight pseudo-walk~$P'$ from~$e_2$ to~$e_1$.
Thus $PP'$ (the concatenation of $P$ and $P'$)  is a closed tight pseudo-walk.
We then define a nearly triangulated plane graph~$D$ such that every vertex of~$D$ corresponds to an edge in~$H$, $PP'$ is on the outer face and any walk in~$D$ corresponds to a tight pseudo-walk in~$H$.
We colour each vertex of~$D$ with the same colour of its corresponding edge in~$H$. 
Since~$e_1$ and~$e_2$ are not in the same monochromatic tight component, there is no red walk in~$D$ from $e_1$ to~$e_2$.
By adapting the proof of Gale~\cite{Gale1979} of the fact that the Hex game cannot end in a draw, one deduces that $H$ contains a blue tight pseudo-walk~$P^*$ from an edge of~$P$ to an edge of~$P'$.
Since $P'$ is contained in~$\bigcup_{j \in [i]} \mathcal{E}(\mathcal{G}_j)$, we have $P^* \subseteq \bigcup_{j \in [i-1]} \mathcal{E}(\mathcal{G}_j)$.
Therefore, $ \varnothing \ne P \cap P^* \subseteq H[W_{i-1}] \cap \bigcup_{j \in [i-1]} \mathcal{E}(\mathcal{G}_j)$ contradicting~\eqref{no_edges_in_W_i}.

We will need the fact that in any not too small induced subgraph of a $(1-\eps, \eps)$-dense $k$-graph there is a short tight pseudo-walk between any two edges.

\begin{proposition} \label{prop:tight_walk_bridge}
    Let $1/n \ll \eps \ll 1/k \leq 1/2$ and let~$H$ be a $(1-\eps, \eps)$-dense $k$-graph on~$n$ vertices and let $W \subseteq V(H)$ be such that $\s{W} > 2\eps n$. Let $e_1, e_2 \in H[W]$. Then there exists a tight pseudo-walk of length~$2k$ from~$e_1$ to~$e_2$ in~$H[W]$.
\end{proposition}
\begin{proof}
    Let $e_1 = x_1 \dots x_k$ and $e_2 = y_1 \dots y_k$. For each $i \in [k-1]$ (in order), let
    \[
        z_i \in N_H(z_1 \dots z_{i-1} x_i \dots x_{k-1}) \cap N_H(z_1 \dots z_{i-1} y_i \dots y_{k-1}) \cap W, 
    \]
    \[
        f_i = z_1 \dots z_i x_i \dots x_{k-1} \text{ and } f'_i = z_1 \dots z_i y_i \dots y_{k-1}.
    \]
    Since~$H$ is $(1-\eps, \eps)$-dense, we have, for every $i \in [k-1]$, 
    \[
        \s{N_H(z_1 \dots z_{i-1} x_i \dots x_{k-1}) \cap N_H(z_1 \dots z_{i-1} y_i \dots y_{k-1}) \cap W} \geq \s{W} - 2\eps n > 0.
    \]
    Thus the~$z_i$ exist. Now $e_1 f_1 \dots f_{k-1} f'_{k-1} \dots f'_1 e_2$ is a tight pseudo-walk of length~$2k$ in~$H$ from~$e_1$ to~$e_2$.
\end{proof} 

Next we show that for any closed tight pseudo-walk~$Q$ in a $(1-\eps, \eps)$-dense $k$-graph~$H$, we can find a nearly triangulated plane graph~$D$ that is a triangulation of~$Q$.

\begin{lemma} \label{lem:triangulation}
    Let $1/n \ll \eps \ll 1/k \leq 1/2$ and let~$H$ be a $(1-\eps, \eps)$-dense $k$-graph on~$n$ vertices. Let $Q = e_1 \dots e_m$ be a closed tight pseudo-walk in~$H$ (of any length~$m$). Then there exists a (finite) nearly triangulated plane graph~$D$ that is a triangulation of~$Q$.
\end{lemma}
\begin{proof}
    First we show that the statement holds for $m \leq 1/{2\eps}$.
    \begin{enumerate}[label=\textbf{Case \Alph*:}, ref=\Alph*, wide, labelwidth=0pt, labelindent=0pt]
    \item \textbf{\boldmath $m \leq 1/{2 \eps}$.\unboldmath} \label{case_1}
    In the following we take the first index of each of the~$S_{i,j}$ and each of the~$e_{i,j}$ (defined below) modulo~$m$.
    \begin{claim} \label{claim:edges}
        There exist edges $e_{i,j} \in H$ for all $(i,j) \in [m] \times [k+1]$ such that the following properties hold. 
        \begin{enumerate}[label = \textup{(\roman*)}]
            \item $e_{i,1} = e_i$ for all $i \in [m]$. \label{edgesi}
            \item $|e_{i,j} \cap e_{i+1, j}|, |e_{i,j} \cap e_{i,j+1}|, |e_{i,j} \cap e_{i-1, j+1}| \geq k-1$ for all $i \in [m]$ and all $j \in [k]$. \label{edgesii}
            \item $e_{i, k+1} = e_{i-1,k+1}$ for all $i \in [m]$. \label{edgesiii}
        \end{enumerate}
    \end{claim}
    \begin{proofclaim}
        For each $i \in [m]$, let $S_{i,0} = e_i$. Note that, since $e_1, \dots, e_m$ is a closed tight pseudo-walk, we have $\s{S_{i,0} \cap S_{i+1,0}} \geq k-1$.
        We show that for each $j \in [k-1]$ in turn, we can for each~$i \in [m]$ choose a subset~$S_{i,j}$ of $S_{i, j-1} \cap S_{i+1, j-1}$ of size~$k-j$ such that for all~$i \in [m]$, $\s{S_{i,j} \cap S_{i+1, j}} \geq k-j-1$. 
        Let $j \geq 1$ and assume that the~$S_{i,j-1}$ for~$i \in [m]$ have already been chosen. 
        For each $i \in [m]$, we let~$S_{i,j}$ be an arbitrary subset of $S_{i,j-1} \cap S_{i+1,j-1}$ of size~$k-j$ (this is possible since $\s{S_{i,j-1} \cap S_{i+1,j-1}} \geq k-j$). 
        Then we have, for each~$i \in [m]$,
        \begin{align*}
            \s{S_{i,j} \cap S_{i+1,j}} &= \s{S_{i,j}} + \s{S_{i+1,j}} - \s{S_{i,j} \cup S_{i+1,j}} \geq 2(k-j) - \s{S_{i+1,j-1}} \\ &= 2(k-j) -(k-j+1) = k -j -1,
        \end{align*}
        as desired, where the inequality follows from the fact $S_{i,j}, S_{i+1,j} \subseteq S_{i+1,j-1}$.

        Since~$H$ is $(1-\eps, \eps)$-dense and $m \leq 1/{2 \eps}$ and thus for any $(k-1)$-sets~$T_i$, $i \in [m]$ such that each~$T_i$ is contained in an edge of~$H$, we have
        \begin{align*}
            \s{\bigcap_{i \in [m]} N_H(T_i)} \geq n -m \cdot \eps n \geq n -n/2 = n/2 >0.
        \end{align*}
        Hence we can construct edges as follows. For each $i \in [m]$, let $S_{i, k} = \varnothing$.
        For each $j \in [k]$, let 
        \begin{align*}
            z_j \in \bigcap_{i \in [m]} N_H(z_1 \dots z_{j-1} \cup S_{i,j}) \text{ and } e_{i, j+1} = z_1 \dots z_j \cup S_{i,j} \in H.
        \end{align*}
        Moreover, for every $i \in [m]$, let $e_{i,1} =e_i = S_{i,0}$ so that we have $e_{i,j} = z_1 \dots z_{j-1} \cup S_{i, j-1}$ for every $i \in [m]$ and $j \in [k+1]$.
				Clearly, \ref{edgesi} and \ref{edgesiii} hold. 

			To see~\ref{edgesii}, consider $i \in [m]$ and $j \in [k]$.
				We have
        \begin{align*}
            \s{e_{i,j} \cap e_{i+1,j}} 
						&= \s{(z_1 \dots z_{j-1} \cup S_{i,j-1}) \cap (z_1 \dots z_{j-1} \cup S_{i+1,j-1})}\\
						& = j-1 + \s{S_{i,j-1} \cap S_{i+1,j-1}} 
						\geq j-1 + k - j  = k-1.
        \end{align*}
        Also,
        \begin{align*}
            \s{e_{i,j} \cap e_{i,j+1}} &= \s{(z_1 \dots z_{j-1} \cup S_{i,j-1}) \cap (z_1 \dots z_{j} \cup S_{i,j})}\\
						& \geq j-1 + \s{S_{i,j-1} \cap S_{i,j}} 
            = j-1 + \s{S_{i,j}} = j-1 + k-j = k-1,
        \end{align*}
        where we recall that $S_{i,j} \subseteq S_{i,j-1}$.
        Finally, we have
        \begin{align*}
            \s{e_{i,j} \cap e_{i-1,j+1}} &= \s{(z_1 \dots z_{j-1} \cup S_{i,j-1}) \cap (z_1 \dots z_{j} \cup S_{i-1,j})}   \\
						&\geq j-1 + \s{S_{i,j-1} \cap S_{i-1,j}} 
             = j-1 + \s{S_{i-1,j}} = j-1 + k-j = k-1,
        \end{align*}
        where we recall that $S_{i-1,j} \subseteq S_{i,j-1}$.
    \end{proofclaim}
    Now let~$e_*$ be the unique edge such that $e_* = e_{i,k+1}$ for all $i \in [m]$ ($e_*$ exists by \cref{claim:edges}\cref{edgesiii}) and note that $|e_{i,k} \cap e_*| \geq k-1$ for all $i \in [m]$ by \cref{claim:edges}\cref{edgesii}.

    Using these edges, we now construct a triangulation of~$Q$.
    Roughly speaking, $D$ will be a plane graph consisting of $k$ concentric circles, where $\{e_{i,1} \colon i \in [m]\}$ is on the outermost circle and $e_*$ is at the centre (see \cref{figure:triangulation}). We now formally define $D$.
    We let~$D$ be the graph with $V(D) = \{(i,j) \colon i \in [m], j \in [k]\} \cup \{*\}$ and
    \begin{align*}
        E(D) &= \{\{(i,j), (i+1, j)\} \colon i \in [m], j \in [k]\} \cup \{\{(i,j), (i, j-1)\} \colon i \in [m], 2 \leq j \leq k\} \\
        &\cup \{\{(i,j), (i+1, j-1)\} \colon i \in [m], 2 \leq j \leq k\} \cup \{\{(i,k), *\} \colon i \in [m]\}.
    \end{align*}
    Define the map $\varphi \colon V(D) \rightarrow E(H)$ by setting $\varphi((i,j)) = e_{i,j}$ for all $i \in [m]$ and $j \in [k]$ and $\varphi(*) = e_*$.
    It follows from \cref{claim:edges}\cref{edgesii} that $\s{\varphi(x) \cap \varphi(y)} \geq k-1$ for all $xy \in E(D)$. 
    
    Next we show that there exists an embedding of~$D$ in the plane such that $(1,1) (2,1) \dots \allowbreak (m,1)$ is the outer face and all other faces are triangles. Consider~$k$ concentric circles $C_1, \dots, C_k$ in the plane~$\mathbb{R}^2$ around the origin~$(0,0)$ with radii $1, \dots, k$, respectively. 
    For each $s \in [m]$, we define the ray 
    \[
        R_s = \left\{\left(t \cos\left(\frac{2 \pi s}{m} \right), t \sin\left(\frac{2 \pi s}{m} \right)\right) \colon t \geq 0 \right\},
    \]
    that is,~$R_s$ is the ray $\{(x,0) \colon x \geq 0\}$ rotated counterclockwise around the origin by $2\pi s/m$ radians.
    Consider the embedding of~$V(D)$ in the plane $\mathbb{R}^2$ that maps~$*$ to~$(0,0)$ and~$(i,j)$ to the intersection of~$C_{k+1-j}$ and~$R_i$ for each $i \in [m]$ and each $j \in [k]$. Then mapping each edge $xy \in E(D)$ to the straight line segment between the images of~$x$ and~$y$ gives the desired embedding of~$D$ in the plane. Hence~$D$ with this embedding forms a triangulation of~$Q$.

    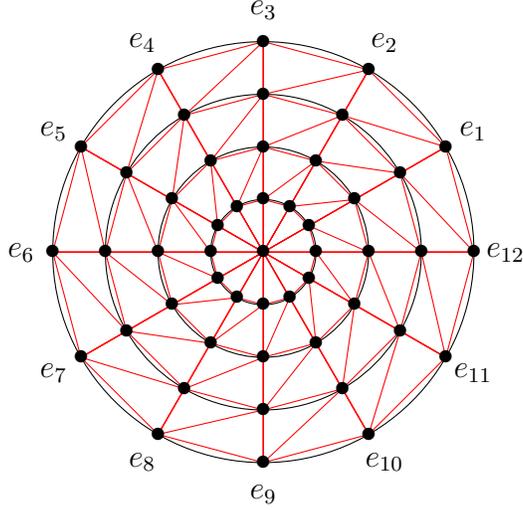
\begin{figure}[htb]
    \centering
    \begin{tikzpicture}[scale = 0.7]
        \foreach \j in {1, ..., 4} {
            \foreach \i in {0, ..., 11}{
                \draw[red, rotate around = {\i*360/12:(0,0)}] (0,0) edge (4,0);
                \draw[red, rotate around = {\i*360/12:(0,0)}] (\j,0) edge (0.866*\j, -0.5*\j);
                \ifthenelse{\j < 4}{\draw[red, rotate around = {\i*360/12:(0,0)}] (\j +1,0) edge (0.866*\j, 0.5*\j);}{}
                \ifthenelse{\j = 4 \and \i > 0}{\draw[rotate around = {\i*360/12:(0,0)}] node at (4.6,0) {$e_{\pgfmathparse{int(\i)}\pgfmathresult}$}; }{\ifthenelse{\j=4}{\draw[rotate around = {\i*360/12:(0,0)}] node at (4.6
                ,0) {$e_{\pgfmathparse{int(12)}\pgfmathresult}$};}{} }
            }
        }
        \foreach \j in {1, ..., 4} {
            \foreach \i in {0, ..., 11}{
                \draw[rotate around = {\i*360/12:(0,0)}] node at (\j,0) {\(\bullet\)};
            }
            \draw (0,0) circle (\j cm);
        }
        \draw node at (0,0) {\(\bullet\)};
    \end{tikzpicture}
    \caption{A triangulation of a closed tight pseudo-walk}
    \label{figure:triangulation}
    \end{figure}
    
    \item \textbf{\boldmath $m > 1/{2 \eps}$.\unboldmath}
    Suppose that there is some $m > 1/{2 \eps}$ for which the statement of the lemma does not hold. Let~$m_*$ be minimal such that there is a closed tight pseudo-walk $Q = e_1 \dots e_{m_*}$ for which the statement does not hold. By Case~\ref{case_1}, we know that $m_* > 1/{2\eps}$. Let $i_* = \lfloor {m_*}/{2} \rfloor$. By \cref{prop:tight_walk_bridge}, there exist edges $f_1, \dots, f_{2k-2}$ such that $e_{i_*} f_1 \dots f_{2k-2} e_1$ is a tight pseudo-walk in~$H$. Note that $Q_1 = e_1 \dots e_{i_*}f_1 \dots f_{2k-2}$ and $Q_2 = e_1 f_{2k-2} f_{2k-3} \dots f_1 e_{i_*} e_{i_* +1} \dots \allowbreak e_{m_*}$ are closed tight pseudo-walks in~$H$ of length less than~$m_*$. Thus there exists plane graphs~$D_1$ and~$D_2$ that are triangulations of~$Q_1$ and~$Q_2$, respectively. It is easy to see that we can combine~$D_1$ and~$D_2$ by `gluing' along the vertices corresponding to the edges $e_{i_*}, f_1, \dots, f_{2k-2}, e_1$ to obtain a plane graph~$D$ that is a triangulation of~$Q$. This is a contradiction to the fact that the statement of the lemma does not hold for~$Q$. \qedhere
    \end{enumerate} 
\end{proof}

\FloatBarrier

The next lemma and its proof was inspired by the proof of the fact that the Hex game cannot end in a draw (see for example~\cite{Gale1979}). It helps us find certain monochromatic walks in nearly triangulated plane graphs whose vertices are coloured with red and blue.

\begin{lemma} \label{lem:grid_argument}
    Let~$D$ be a (finite) nearly triangulated plane graph with a red-blue colouring of its vertices. Let $x_1 \dots x_m$ be the outer face of~$D$ where the indices of the~$x_i$ are taken modulo~$m$. Suppose that~$x_1$ is red and~$x_2$ is blue. Then there exists some $i_* \in [m]$ such that~$x_{i_*}$ is blue,~$x_{i_*+1}$ is red, and there exists a red walk in~$D$ from~$x_1$ to~$x_{i_*+1}$ and a blue walk in~$D$ from~$x_2$ to~$x_{i_*}$. 
\end{lemma}
\begin{proof}
    Consider the dual graph~$D^*$ of~$D$. The set of vertices of~$D^*$ is the set~$F$ of faces of~$D$ and $f_1f_2 \in D^*$ if and only if the faces~$f_1$ and~$f_2$ share an edge $e \in D$. We say that the edge~$f_1f_2$ of~$D^*$ corresponds to the edge~$e$ of~$D$. Let~$f_*$ be the outer face of~$D$. Note that for every face $f \in F \setminus \{f_*\}$, we have $d_{D^*}(f) = 3$ since~$f$ is a triangle in~$D$. Now consider the spanning subgraph $\widetilde{D}$ of~$D^*$ which contains only the edges $f_1f_2 \in D^*$ such that the corresponding edge $e \in D$ consists of two vertices with different colours. 
    Note that for every face $f \in F \setminus\{f_*\}$, we have $d_{\widetilde{D}}(f) \in \{0,2\}$. Thus there is a unique cycle~$C$ in $\widetilde{D}$ that contains~$f_*$ and the edge of $\widetilde{D}$ that corresponds to the edge $x_1x_2 \in D$. The other edge of~$C$ incident to~$f_*$ corresponds to an edge $x_{i_*}x_{i_*+1}$ of~$D$ for some $i_* \in [m]$. Observe that there exists a closed curve $\mathcal{C}$ in the plane that starts on the outer face and only crosses edges of~$D$ that correspond to edges of~$C$. Each edge crossed by $\mathcal{C}$ consists of two vertices of different colours. Moreover, since all faces except the outer face are triangles, all the vertices of crossed edges of~$D$ of one colour are in the bounded region of~$\mathbb{R}^2 \setminus \mathcal{C}$ and all the vertices of crossed edges of the other colour are in the unbounded region of $\mathbb{R}^2 \setminus \mathcal{C}$. Hence~$x_{i_*}$ is blue and~$x_{i_*+1}$ is red. Furthermore, there is a red path in~$D$ from~$x_1$ to~$x_{i_*+1}$ and  a blue path in~$D$ from~$x_2$ to~$x_{i_*}$. See \cref{fig:grid_argument}.
\end{proof}

\begin{figure}[htb]
    \centering
    \begin{tikzpicture}
        \foreach \pos/\name in {{(6,0)/x_1},{(3.75,3)/x_4},{(0.5,2.5)/x_6},{(-0.5,1.5)/x_7},{(3.75,-1.5)/x_{11}},{(5,-1)/x_{12}},{(1.8,-0.4)/z_1},{(0.6,0.2)/z_2},{(0.7,1.3)/z_3},{(4.1,1.7)/a}} {
                \node[red_vertex] (\name) at \pos {}; 
        }
        \foreach \pos/\name in {{(6,1.5)/x_2},{(5,2.5)/x_3},{(2,3)/x_5},{(-0.5,0)/x_8},{(0.5,-1)/x_9},{(2,-1.5)/x_{10}},{(5,0.75)/y_1},{(3.75,0)/y_2},{(2,0.75)/y_3},{(3,1.8)/y_4}} {
                \node[blue_vertex] (\name) at \pos {};
        }
				
				\foreach \i in {x_1,x_2,x_3,x_4,x_5,x_6,x_7,x_8,x_9,x_{10},x_{11},x_{12}}
					{				\draw (\i) node {$\i$};	}

        \foreach \source/\dest in {x_1/x_2,x_2/x_3,x_3/x_4,x_4/x_5,x_5/x_6,x_6/x_7,x_7/x_8,x_8/x_9,x_9/x_{10},x_{10}/x_{11},x_{11}/x_{12},x_{12}/x_1,x_1/y_1,x_2/y_1,x_3/y_1,a/y_1,y_1/y_2,x_{12}/y_1,a/y_2,y_2/x_{12},y_2/x_{11},a/x_3,a/x_4,a/y_4,y_2/y_4,y_2/y_3,y_2/z_1,x_{11}/z_1,z_1/x_{10},z_1/x_9,z_1/z_2,z_2/x_9,z_2/x_8,z_3/x_8,z_3/x_7,z_2/z_3,z_3/x_6,y_3/z_1,y_3/z_2,y_3/z_3,y_3/y_4,z_3/y_4,y_4/x_5,z_3/x_5,y_4/x_4} {
                \path[edge] (\source) -- (\dest);
        }
        \begin{pgfonlayer}{bg}
            \foreach \source/\dest in {x_1/x_{12},x_{12}/x_{11},x_{11}/z_1,z_1/z_2,z_2/z_3,z_3/x_6} {
              \path[highlighted_edge, red!15] (\source.center) -- (\dest.center);
            }
            \foreach \source/\dest in {x_2/y_1,y_1/y_2,y_2/y_3,y_3/y_4,y_4/x_5} {
              \path[highlighted_edge, blue!15] (\source.center) -- (\dest.center);
            }

        \end{pgfonlayer}

        \draw [violet, line width =1.2pt, tension= 0.7] plot [smooth cycle] coordinates {(6.7,0.75)(5.6,0.75) (5.4,0) (4.5,0) (4.1,-0.8) (3.2,-0.6) (2.5,0.1)(1.5,0.15) (1.1,0.75) (1.9,1.3) (1.9,2.1) (1.1,2.3) (1.3, 3.3)  (3,3.9) (5, 3.5)  (6.75, 2)};

        \node[violet] at (7.2, 0.75) {$\mathcal{C}$};
    \end{tikzpicture}
    \caption{A picture of the argument for \cref{lem:grid_argument} with $m=12$ and $i^*=5$.} \label{fig:grid_argument}
\end{figure}

The next lemma allows us to deduce that certain edges are in the same monochromatic tight component.

\begin{lemma} \label{lem:blue_tight_walk_2}
    Let $1/n \ll \eps \ll 1/k \leq 1/2$. Let~$H$ be a $2$-edge-coloured $(1-\eps, \eps)$-dense $k$-graph on~$n$ vertices. Let $Q= e_1 \dots e_m$ be a closed tight pseudo-walk in~$H$ such that for some $1 < i < m$, we have that~$e_1$ and~$e_i$ are edges in different red tight components of~$H$. Then there exists a blue edge in $\{e_2 \dots, e_{i-1}\}$ that is in the same blue tight component as a blue edge in $\{e_{i+1}, \dots, e_m\}$. The statement also holds with colours reversed.
\end{lemma}
\begin{proof}
    Let~$s$ be the number of red edges in $\{e_2, \dots, e_{i-1}\}$. We proceed by induction on~$s$.
    \begin{enumerate}[label=\textbf{Case \Alph*:}, ref=\Alph*, wide, labelwidth=0pt, labelindent=0pt, parsep=5pt]
    \item \textbf{\boldmath $s=0$.\unboldmath}\label{CaseA}
    By \cref{lem:triangulation}, there exists a nearly triangulated plane graph~$D$ that is a triangulation for~$Q$. That is, there exists a map $\varphi \colon V(D) \rightarrow E(H)$ such that if $xy \in E(D)$, then $\s{\varphi(x) \cap \varphi(y)} \geq k-1$ and the outer face of~$D$ is $x_1 \dots x_m$ where $\varphi(x_i) = e_i$ for each~$i \in [m]$. 
    We colour the vertices of~$D$ with red and blue by giving each $x \in V(D)$ the colour of the edge $\varphi(x) \in H$. Note that any red (or blue) walk $y_1 \dots y_\ell$ in~$D$ corresponds to a red (or blue) tight pseudo-walk $\varphi(y_1) \dots \varphi(y_\ell)$ in~$H$. Let the indices of the~$x_i$ be taken modulo~$m$. By \cref{lem:grid_argument}, there exists some $i_* \in [m]$ such that~$x_{i_*}$ is blue,~$x_{i_*+1}$ is red, and there exists a red walk in~$D$ from~$x_1$ to~$x_{i_*+1}$ and a blue walk in~$D$ from~$x_2$ to~$x_{i_*}$. Note that since $e_2, \dots, e_{i-1}$ and thus $x_2, \dots, x_{i-1}$ are blue, we have $i-1 \leq i_* \leq m$. Suppose $i_* = i-1$. Then there is a red walk in~$D$ from~$x_1$ to~$x_i$. This implies that there is a red tight pseudo-walk in~$H$ from~$e_1$ to~$e_i$, contradicting the fact that~$e_1$ and~$e_i$ are in different red tight components of~$H$. Hence $i \leq i_* \leq m$. Since~$x_i$ is red and~$x_{i_*}$ is blue, we have $i+1 \leq i_* \leq m$. Since there is a blue walk in~$D$ from~$x_2$ to~$x_{i_*}$, there is a blue tight pseudo-walk from~$e_2$ to $e_{i_*} \in \{e_{i+1}, \dots, e_m\}$.

    \item \textbf{\boldmath $s \ge 1$.\unboldmath}
    Let $j_2 \in \{3, \dots, i\}$ be minimal such that~$e_{j_2}$ is a red edge in a different red tight component than~$e_1$. Let $j_1 \in \{1, \dots, j_2 -2\}$ be maximal such that~$e_{j_1}$ is red. Note that~$e_{j_1}$ is in the same red tight component as~$e_1$ and thus~$e_{j_1}$ is in a different red tight component than~$e_{j_2}$. Moreover, note that all edges in $\{e_{j_1+1}, \dots, e_{j_2-1}\}$ are blue. Hence by Case~\ref{CaseA}, there exists a blue tight pseudo-walk~$P_*$ from~$e_{j_1+1}$ to a blue edge $f \in \{e_{j_2 +1}, \dots, e_m, e_1, \dots, e_{j_1-1}\}$. Let $P_* = e_{j_1+1} f_1 \dots f_s f$.
    If~$f$ is not in $\{e_1, \dots, e_i\}$, then we are done. So we may assume that~$f$ is in $\{e_2, \dots, e_{j_1-1}, e_{j_2+1}, \dots e_{i-1}\}$ (note that~$f$ cannot be~$e_1$ or~$e_i$ as these are red edges). 
    Suppose that $f = e_{i'}$ for some $i' \in \{2, \dots, j_1-1\}$. Let~$Q_*$ be the closed tight pseudo-walk obtained from $Q$ by replacing $e_{i'} \dots e_{j_1+1}$ with $P_*$, that is $Q_* = e_1 \dots e_{i'} f_s f_{s-1} \dots f_1 e_{j_1+1}\dots e_m$. Note that $Q_*$ has fewer red edges between $e_1$ and $e_i$. Hence we are done by induction. A similar argument holds when $f \in \{e_{j_2+1}, \dots, e_{i-1}\}$. \qedhere
    \end{enumerate} 
\end{proof}

We will need the fact that any almost complete red-blue edge-coloured $k$-graph contains a monochromatic tight component that covers almost all vertices. 
This easily follows from the next proposition.

\begin{proposition}[Corollary 26 in \cite{Lo2023}] \label{prop:large_MC}
    Let $1/n \ll \eps \leq 1/324$. Let~$F$ be a $2$-edge-coloured $2$-graph with $\s{V(F)} \leq n$ and $\s{E(F)} \geq (1-\eps)\binom{n}{2}$. Then there exists a subgraph~$F'$ of~$F$ of order at least $(1-3\sqrt{\eps})n$ that contains a spanning monochromatic component.
\end{proposition}

We say that a vertex~$v$ in a $2$-edge-coloured $k$-graph~$H$ is \emph{spanned} by a monochromatic tight component~$K$ of~$H$ if there is an edge $e \in K$ such that $v \in e$.

\begin{proposition} \label{prop:Almost_spanning_MTC}
    Let $1/n \ll \eps \ll 1/k \leq 1/2$ and let~$H$ be a $(1-\eps, \eps)$-dense $k$-graph on~$n$ vertices. Then~$H$ contains a monochromatic tight component~$K$ of~$H$ spanning at least $(1-\sqrt{\eps})n$ vertices of~$H$.
\end{proposition}
\begin{proof}
    Let $S \in \binom{V(H)}{k-2}$ with $d_H(S) \geq (1-\eps)\binom{n}{2}$ (the set~$S$ exists since~$H$ is $(1-\eps, \eps)$-dense). By \cref{prop:large_MC}, the link graph~$H_S$ contains a monochromatic component of order at least $(1-3\sqrt{\eps})n$. The monochromatic tight component~$K$ of~$H$ that contains the edges in $\{e \cup S \colon e \in H_S\}$ has the desired property.
\end{proof}

\subsection{\texorpdfstring{Proof of \cref{lem:main}}{Proof of the main lemma}}

\begin{proof}[Proof of $\cref{lem:main}$]
    We define the graph~$G$ whose vertex set~$V(G)$ is the set of monochromatic tight components of~$H$ and $F_1 F_2 \in E(G)$ if and only if the monochromatic tight components~$F_1$ and~$F_2$ have different colours and there are edges $e_1 \in F_1$ and $e_2 \in F_2$ such that $\s{e_1 \cap e_2} = k-1$. 
    By \cref{prop:Almost_spanning_MTC}, there exists a monochromatic tight component~$F_*$ of~$H$ that spans at least $(1-3\sqrt{\eps})n$ vertices of~$H$. Without loss of generality assume that~$F_*$ is red.
    \begin{claim} \label{claim:no_edges_far}
        There are at most $2\sqrt{\eps} n^k$ edges of~$H$ that are in monochromatic tight components of~$H$ at distance greater than~$k-1$ from~$F_*$ in~$G$.
    \end{claim}
    \begin{proofclaim}
        We call an edge $e \in H$ \emph{bad} if it is in a monochromatic tight component of~$H$ at distance greater than~$k-1$ from~$F_*$ in~$G$.
        There are at most $(3\sqrt{\eps}n)^k \leq \sqrt{\eps}n^k$ edges of~$H$ that do not contain a vertex spanned by~$F_*$. 
        Thus it suffices to show that each vertex~$v$ that is spanned by~$F_*$ is contained in at most $\sqrt{\eps}n^{k-1}$ bad edges of~$H$. 
        Let~$v$ be a vertex that is spanned by~$F_*$ and let $e = vx_1 \dots x_{k-1} \in F_*$. 
        Since~$H$ is $(1-\eps, \eps)$-dense, for each $1 \leq j \leq k-2$, there are at most $\eps \binom{n}{j}n^{k-j-1} \leq \eps n^{k-1}$ many edges $f = vy_1 \dots y_{k-1}$ of~$H$ containing~$v$ such that $e_j(f) = v y_1 \dots y_{j} x_{j+1} \dots x_{k-1} \notin H$. Hence for all but at most $(k-2) \eps n^{k-1} \leq \sqrt{\eps}n^{k-1}$ many edges~$f$ of~$H$ that contain~$v$, we have that $e e_1(f) \dots e_{k-2}(f) f$ is a tight pseudo-walk of length~$k$ from~$e$ to~$f$ in~$H$ which implies that~$f$ is not bad. Hence~$v$ is contained in at most $\sqrt{\eps}n^{k-1}$ many bad edges of~$H$, as desired.
    \end{proofclaim}
    Let $\mathcal{G}_1 = \{F_*\}$ and for each $2 \leq i \leq k$, let
    \[
        \mathcal{G}_i = \left( \bigcup_{F \in \mathcal{G}_{i-1}} N_G(F)\right) \setminus \bigcup_{j \in [i-1]} \mathcal{G}_j \quad\text{ and }\quad \mathcal{E}(\mathcal{G}_i) = \bigcup_{F \in \mathcal{G}_i} F,
    \]
    that is, $\mathcal{G}_i$ is the set of monochromatic tight components of~$H$ that are at distance~$i-1$ from~$F_*$ in~$G$ and $\mathcal{E}(\mathcal{G}_i)$ is the set of edges of~$H$ that are in one of these monochromatic tight components. 
    Note that, since~$F_*$ is red, if~$i$ is odd then the monochromatic tight components in $\mathcal{G}_i$ and all edges in $\mathcal{E}(\mathcal{G}_i)$ are red (and if~$i$ is even they are blue). Let $W_0 = V(H)$. For $i = 1, \dots, k$ in turn, let~$M_i$ be a maximal matching in~$H[W_{i-1}] \cap \mathcal{E}(\mathcal{G}_i)$ and let $W_i = W_{i-1} \setminus V(M_i)$. Observe that
    \begin{equation} \label{eq:no_edges}
        H[W_i] \cap \bigcup_{j \in [i]} \mathcal{E}(\mathcal{G}_j) = \varnothing \quad \text{for each } i \in [k].
    \end{equation}

    \begin{claim} \label{claim:W_k_small}
        We have $\s{W_{k}} \leq \eta n$.
    \end{claim}
    \begin{proofclaim}
        Since~$H[W_{k}]$ does not contain any edges in $\bigcup_{j \in [k]} \mathcal{E}(\mathcal{G}_j)$,~$H[W_{k}]$ contains only edges of monochromatic tight components that are at distance greater than~$k-1$ from~$F_*$ in~$G$. Hence, by \cref{claim:no_edges_far}, we have that $\s{H[W_{k}]} \leq 2\sqrt{\eps} n^k$. Since~$H$ is $(1-\eps, \eps)$-dense, we have 
        \[
            \s{H[W_{k}]} \geq \left(\binom{\s{W_{k}}}{k-1}-\eps \binom{n}{k-1}\right)(\s{W_{k}} - \eps n) \geq \frac{\s{W_{k}}^k}{k^k} - 2\eps n^k.
        \]
        It follows that $\s{W_{k}} \leq (3 \sqrt{\eps} k^k n^k)^{1/k} \leq \eta n$.
    \end{proofclaim}
    
    \begin{claim} \label{claim:One_MTC}
        For each $i \in [k]$,~$M_i$ contains only edges from one monochromatic tight component or $\s{W_{i-1}} \leq 2 \eps n$.
    \end{claim}
    \begin{proofclaim}
        For $i =1$, this trivially holds. Let $2 \leq i \leq k$ and suppose that $\s{W_{i-1}} > 2 \eps n$. Let~$c$ be the colour of the monochromatic tight components in $\mathcal{G}_i$ and let~$\tilde{c}$ be the opposite colour. Suppose for a contradiction that there exist two edges~$e_1$ and~$e_2$ in~$H[W_{i-1}]$ that are in different monochromatic tight components of~$\mathcal{G}_i$. 
        By \cref{prop:tight_walk_bridge}, there exists a tight pseudo-walk~$P$ in~$H[W_{i-1}]$ from~$e_1$ to~$e_2$. 
				Moreover, by the definition of the $\mathcal{G}_i$, there exists a tight pseudo-walk~$P'$ from~$e_2$ to~$e_1$ that only uses edges in $\bigcup_{j \in [i]} \mathcal{E}(\mathcal{G}_j)$. Note that $Q = P P'$ is a closed tight pseudo-walk. By \cref{lem:blue_tight_walk_2}, we have that there exists an edge~$e$ in~$P$ of colour~$\tilde{c}$ that is in the same monochromatic tight component as an edge in~$P'$. 
        Since all edges of $P'$ are in $\bigcup_{j \in [i]} \mathcal{E}(\mathcal{G}_j)$ and all edges in $\mathcal{E}(\mathcal{G}_i)$ are of colour~$c$, we have that~$e$ is in $\bigcup_{j \in [i-1]} \mathcal{E}(\mathcal{G}_j)$. 
        Therefore, $e \in P \cap \bigcup_{j \in [i-1]} \mathcal{E}(\mathcal{G}_j) \subseteq H[W_{i-1}] \cap \bigcup_{j \in [i-1]} \mathcal{E}(\mathcal{G}_j)$, a contradiction to \cref{eq:no_edges}.
    \end{proofclaim}
    Let $i^* \in [k]$ be minimal such that $\s{W_{i^*}} \leq \eta n$ and let $M = \bigcup_{j \in [i^*]} M_i$. By \cref{claim:W_k_small},~$i^*$ exists. Since $\eps \ll \eta$, by \cref{claim:One_MTC},~$M$ contains only edges from $i^* \leq k$ monochromatic tight components. Moreover, $\s{V(H) \setminus V(M)} = \s{W_{i^*}} \leq \eta n$. Hence~$M$ is the desired matching.
\end{proof}

\section{\texorpdfstring{Proof of \cref{thm:main}}{Proof of the main theorem}} \label{section:end}

To prove \cref{thm:main} we use the following hypergraph version of {\L}uczak's connected matching method that reduces the problem of finding monochromatic tight cycles in $2$-edge-coloured complete $k$-graphs to finding matchings in monochromatic tight components in $2$-edge-coloured almost complete $k$-graphs.

\begin{corollary}[{\cite[Corollary~20]{Lo2023}}] 
\label{cor:matchings_to_cycles}
Let $1/n \ll 1/m \ll \eps \ll \eta \ll \gamma, 1/k,1/s$ with $k \ge 3$. 
Suppose that every $2$-edge-coloured $(1-\eps, \eps)$-dense $k$-graph~$H$ on~$m$ vertices contains a matching in~$H$ that covers all but at most $\eta m$ vertices of~$H$ and only contains edges from at most~$s$ monochromatic tight components of~$H$.
Then any $2$-edge-coloured complete $k$-graph on~$n$ vertices contains~$s$ vertex-disjoint monochromatic tight cycles covering at least $(1-\gamma) n$ vertices.
\end{corollary}

Now \cref{thm:main} follows immediately from \cref{lem:main} and \cref{cor:matchings_to_cycles}. 

We remark that it is easy to see from the proofs of \cref{lem:main} and \cref{cor:matchings_to_cycles} that in \cref{thm:main} we could additionally require that at most $\lceil k/2 \rceil$ of the~$k$ monochromatic tight cycles have the same colour.

\section*{Acknowledgments}
The research leading to these results was supported by EPSRC, grant no. EP/V002279/1 and EP/V048287/1 (Allan Lo).
There are no additional data beyond that contained within the main manuscript.
This project has received partial funding from the European Research
Council (ERC) under the European Union's Horizon 2020 research and innovation programme (grant agreement no.\ 786198, Vincent Pfenninger). This research was also supported in part by FWF 10.55776/I6502 (Vincent Pfenninger). For the purpose of open access, the authors have applied a CC BY public copyright licence to any Author Accepted Manuscript version arising from this submission.

\bibliographystyle{abbrv}
\bibliography{bibliography}

\end{document}